\documentclass[10pt]{amsart}
\usepackage{amsfonts,amsmath,amscd,amssymb,amsthm,mathrsfs,esint}
\usepackage{color,fancyhdr,framed,graphicx,verbatim,mathtools}
\usepackage[margin=1in]{geometry}
\usepackage{tikz}
\usepackage{comment}
\usetikzlibrary{matrix,arrows,decorations.pathmorphing}
\newcommand{\R}[0]{\mathbb{R}}
\newcommand{\C}[0]{\mathbb{C}}
\newcommand{\N}[0]{\mathbb{N}}

\newcommand{\Z}[0]{\mathbb{Z}}

\newcommand{\calH}[0]{\mathcal{H}}
\newcommand{\calP}[0]{\mathcal{P}}
\newcommand{\tn}[1]{\text{#1}}

\DeclareMathOperator{\spn}{span}
\DeclareMathOperator{\Sym}{Sym}
\DeclareMathOperator{\Skew}{Skew}
\DeclareMathOperator{\divides}{div}

\newcommand{\lrfl}[1]{\left\lfloor#1\right\rfloor}

\newcommand{\lrp}[1]{\left ( {#1} \right )}

\newtheorem{theorem}{Theorem}[section]
\theoremstyle{definition}
\newtheorem{corollary}[theorem]{Corollary}
\newtheorem{lemma}[theorem]{Lemma}
\newtheorem{proposition}[theorem]{Proposition}

\newtheorem{definition}[theorem]{Definition}

\newtheorem{remark}[theorem]{Remark}

\newtheorem*{proposition*}{Proposition}

\title{Spectral Analysis of the Kohn Laplacian on Lens Spaces}

\author{Colin Fan} 
\address[Colin Fan]{Rutgers University--New Brunswick, Department of Mathematics, Piscataway, NJ 08854, USA}
\email{colin.fan@rutgers.edu}

\author{Elena Kim}
\address[Elena Kim]{Massachusetts Institute of Technology, Department of Mathematics, Cambridge, MA 02142, USA}
\email{elenakim@mit.edu}

\author{Zoe Plzak} 
\address[Zoe Plzak]{Department of Mathematics, Boston University, Boston, MA 02215, USA}
\email{zplzak@bu.edu}

\author{Ian Shors} 
\address[Ian Shors]{Department of Mathematics, Harvey Mudd College, Claremont, CA 91711, USA}
\email{ishors@g.hmc.edu }

\author{Samuel Sottile} 
\address[Samuel Sottile]{Department of Mathematics, Michigan State University, East Lansing, MI 48910, USA}
\email{sottile1@msu.edu}

\author{Yunus E. Zeytuncu}
\address[Yunus E. Zeytuncu]{University of Michigan--Dearborn, Department of Mathematics and Statistics, 
Dearborn, MI 48128, USA}
\email{zeytuncu@umich.edu}

\thanks{This work is supported by NSF (DMS-1950102 and DMS-1659203). The work of the second author is suppported by a grant by the NSF Graduate Research Fellowship   (\#1745302) and the work of the last author is also partially supported by a grant from the Simons Foundation (\#353525).}


\begin{document}
\begin{abstract}
We obtain an analog of Weyl's law for the Kohn Laplacian on lens spaces.
We also show that two 3-dimensional lens spaces with fundamental groups of equal prime order are isospectral with respect to the Kohn Laplacian if and only if they are CR isometric.
\end{abstract}

\maketitle

%\tableofcontents

\setcounter{section}{-1}

\section{Introduction}
One of the most celebrated results in spectral geometry is due to Hermann Weyl in 1911: given any bounded domain $\Omega$ in $\mathbb{R}^n$, one can obtain its volume by understanding the eigenvalue distribution of the Laplacian on $\Omega$. This result, usually referred to as Weyl's law, has since been extended to the Laplace-Beltrami operator and other elliptic operators on Riemannian manifolds \cite{Stanton}.

However, as the Kohn Laplacian $\Box_b$ is not elliptic, it is an open problem to find an analog of Weyl's law for the Kohn Laplacian on functions for CR manifolds.
Stanton and Tartokoff presented in \cite{Stanton1984TheHE} a version 
%on $q$-forms where $1 \leq q${\color{red} 
on $q$-forms that are not functions or of top-degree. Motivated by this problem, in \cite{REU2020Weyl}, the authors prove an analog of Weyl's law for the Kohn Laplacian on functions on odd-dimensional spheres. Moreover, they conjecture that their result generalizes to compact strongly pseudoconvex embedded CR manifolds of hypersurface type.

In this paper, we first generalize the aforementioned result to lens spaces. In particular, we show that one can hear the order of its fundamental group. Moreover, the universal constant obtained in the asymptotic behavior of the eigenvalue counting function for the Kohn Laplacian matches the universal constant computed on spheres, providing further evidence for the conjecture above. 

The second portion of this paper is inspired by the 1979 result of Ikeda and Yamamoto \cite{ikeda1979} which shows that if two 3-dimensional lens spaces with fundamental group of order $k$ are isospectral with respect to the standard Laplacian, then they are isometric as Riemannian manifolds. That is, one can hear the shape of a lens space from its standard Laplacian.
We investigate the analogous question: does the spectrum of the Kohn Laplacian determine when two lens spaces are CR isometric? We say that two CR manifolds are CR isospectral if the spectrums of the Kohn Laplacian on each manifold are identical.  We suspect that, as in the Riemannian setting, one can hear the shape of a 3-dimensional lens space as a CR manifold from its Kohn Laplacian. As a partial result towards this conjecture, we show that two isospectral lens spaces must have fundamental groups of equal order. Moreover, we show that if the order is prime, and if the dimension is $3$, being CR isospectral is equivalent to being CR isometric.

The organization of this paper is as follows. We first prove the analog of Weyl's law for the Kohn Laplacian on lens spaces. The proof follows from a careful counting of solutions to a system of diophantine equations and asymptotic analysis of binomial coefficients. Next, we prove the result on isospectral lens spaces by using a chain of equivalencies and extending some techniques of Ikeda and Yamamoto.
\section{Preliminaries}

Let $S^{2n-1}$ denote the unit sphere in $\mathbb{C}^n$. We begin by formally defining lens spaces. 

\begin{definition}\label{def:lens_space}
Let $k$ be a positive integer and let $\ell_1,\ldots, \ell_n$ be integers relatively prime to $k$. Let $\zeta = e^{2\pi i/k}$. Define the map $g: S^{2n-1}\to S^{2n-1}$ by 
\[g\left(z_1,\ldots, z_n\right) = \left(\zeta^{\ell_1} z_1,\ldots,\zeta^{\ell_n} z_n\right). \]
Let $G$ be the cyclic subgroup of the unitary group $U\left(n\right)$ generated by $g$. The \emph{lens space} $L\left(k;\ell_1,\ldots,\ell_n\right)$ is the quotient space, $G\setminus S^{2n-1}$.
\end{definition}

We now consider some preliminary results that are necessary in our analysis of the spectrum of the Kohn Laplacian on lens spaces. First we present Folland's calculation of the spectrum of $\Box_b$ on spheres.

\begin{theorem}[\cite{Folland}]\label{thm:spectrum_box_b} The space of square integrable functions on $S^{2n-1}$ yields the following spectral decomposition for $\square_b$,
\[L^2\left(S^{2n-1}\right) =\bigoplus_{p,q\geq 0} \calH_{p,q}\left(S^{2n-1}\right),\]
where $\calH_{p,q}\left(S^{2n-1}\right)$ denotes the space of harmonic polynomials on the sphere of bidegree $\left(p,q\right)$. Furthermore, $\calH_{p,q}\left(S^{2n-1}\right)$ has corresponding eigenvalue $2q\left(p+n-1\right)$ and
\begin{align*}
\dim \calH_{p,q} \left(S^{2n-1}\right) 
%{\color{red}\binom{n + p - 1}{p} \binom{n + q - 1}{q} - \binom{n + p - 2}{p - 1}\binom{n + q - 2}{q - 1}}\\
=\left( \frac{p+q}{n-1}+1\right)\binom{p+n-2}{n-2}\binom{q+n-2}{n-2}.
\end{align*}

\end{theorem}

%{\color{red} We use both forms of $\dim \mathcal{H}_{p,q}\left(S^{2n-1}\right)$ in this paper.} 

For the remainder of the paper we will write $\mathcal{H}_{p,q}$ for $\calH_{p,q}\left(S^{2n-1}\right)$. Note that Theorem \ref{thm:spectrum_box_b} allows us to quickly compute the spectrum of $\Box_b$ on $S^{2n-1}$; each eigenvalue is a non-negative even integer and the multiplicity of an eigenvalue $\lambda$ is  
\[\sum_{2q(p+n-1) = \lambda} \dim\calH_{p,q}.\] 

It is also important to note that a basis for $\mathcal{H}_{p,q}$ can be computed explicitly. 
\begin{theorem}[\cite{REU18}]
For $\alpha \in \mathbb{N}^n$, define $\left|\alpha\right| = \sum_{j=1}^n\alpha_j$. For $\alpha,\beta \in\mathbb{N}^n$ let
    \[\overline{D}^{\alpha} = \frac{\partial^{\left|\alpha\right|}}{\partial \overline{z}_1^{\alpha_1}\cdots \partial \overline{z}_n^{\alpha_n}} \quad \text{and} \quad \ D^{\beta} = \frac{\partial^{\left|\beta\right|}}{\partial z_1^{\beta_1}\cdots \partial z_n^{\beta_n}}. \]
The following yields a basis for $\mathcal{H}_{p,q} \left(S^{2n-1}\right)$:
    \[
    \left\{\overline{D}^{\alpha}D^{\beta}|z|^{2-2n} :\ |\alpha|=p,\ |\beta|=q,\ \alpha_1=0\textnormal{ or }\beta_1=0\right\}.\]
\end{theorem}

Note that $G$ acts naturally on $L^2\left(S^{2n-1}\right)$ by precomposition. In particular, note that each $\calH_{p,q}$ space is invariant under $G$. It follows that, 

\begin{theorem} 
The space of square integrable functions on $L(k;\ell_1,\dots,\ell_n)$ has the following spectral decomposition for $\square_b$ on $L(k;\ell_1,\dots,\ell_n)$,
\[L^2\left(L\left(k;\ell_1,\dots,\ell_n\right)\right) = \bigoplus_{p,q\geq 0} \calH_{p,q}\left(L\left(k;\ell_1,\dots,\ell_n\right)\right) \cong \bigoplus_{p,q\geq 0} \calH^G_{p,q},\]
where $\calH_{p,q}^G$ is the subspace of $\calH_{p,q}$ consisting of elements invariant under the action of $G$. That is,
\[\calH_{p,q}^G = \left\{f \in \calH_{p,q} : f \circ g = f\right\}.\]
\end{theorem}

Just as on the sphere, each $\calH_{p,q}^G$ (presuming it is non-trivial) is an eigenspace of $\Box_b$ with eigenvalue $2q\left(p+n-1\right)$. It follows that the multiplicity of an eigenvalue $\lambda$ in the spectrum of $\Box_b$ on the lens space $G \setminus S^{2n-1}$ is given by 
\[\sum_{2q\left(p+n-1\right) = \lambda} \dim\calH_{p,q}^G,\]
where we further restrict $p,q$ such that $\calH_{p,q}^G \neq \left\{0\right\}$.
This observation gives the following proposition.
%where the sum ranges over all pairs $(p,q)$ such that $2q(p+n-1) = \lambda$.

\begin{proposition} \label{prop:system}
The dimension of $\calH^G_{p,q}$ is equal to the number of solutions $\left(\alpha,\beta\right)$ to the system
\begin{align*}
\left|\alpha\right| = p, \quad \left|\beta\right| = q,\\
\alpha_1 = 0 \quad \text{or} \quad \beta_1 = 0,\\
\sum_{j = 1}^{n} \ell_j\left(\alpha_j - \beta_j\right) \equiv 0\mod k,
\end{align*}
where $\alpha=\left(\alpha_1,\ldots,\alpha_n\right)$ and $\beta = \left(\beta_1,\ldots,\beta_n\right)$ are $n$-tuples of nonnegative integers and $\left|\cdot\right|$ denotes the sum of these integers.
\end{proposition}
\begin{proof}
    Note that the first two conditions on the system are given by the basis of $\mathcal{H}_{p,q}$. Now let $f_{\alpha,\beta} =     \overline{D}^{\alpha}D^{\beta}|z|^{2-2n}$. Recall that $f_{\alpha,\beta} \in \mathcal{H}_{p,q}^G$ if and only if $ f_{\alpha,\beta} \circ g = f_{\alpha,\beta}$. By the chain rule and the fact that $g$ can be thought of as a unitary matrix,
    \begin{align*}
    f_{\alpha,\beta} \left(w\right) &= \left(\overline{D}^\alpha D^\beta \left|z\right|^{2- 2n}\right)\big|_{z=w}\\
    &= \left(\overline{D}^\alpha D^\beta \left| g z\right|^{2- 2n}\right)\big|_{z=w}\\
    &= \zeta^{\sum_{j=1}^n \ell_j \left(\beta_j - \alpha_j\right)} \left(\overline{D}^\alpha D^\beta \left|z^{2 - 2n}\right|\right)\big|_{z=gw}\\
    &=\zeta^{\sum_{j=1}^n \ell_j \left(\beta_j - \alpha_j\right)} g \circ f_{\alpha,\beta}\left(w\right).
 \end{align*}
    Since the collection of the $f_{\alpha,\beta}$ comprise of a basis for $\calH_{p,q}$ on which $g$ acts diagonally, the dimension of the $G$-invariant subspace of $\calH_{p,q}$ is simply the number of these basis vectors that are fixed by $g$. Thus, $f_{\alpha,\beta}\circ g = f_{\alpha,\beta}$ if and only if the last condition in the proposition holds.
\end{proof}

\section{Results}

We now outline the main results of this paper. For $\lambda>0$, let $N\left(\lambda\right)$ denote the number of positive eigenvalues, including multiplicity, of $\square_b$ on $L^2\left(S^{2n-1}\right)$ that are at most $\lambda$. Similarly, let $N_L\left(\lambda\right)$ denote the number of positive eigenvalues, including multiplicity, of $\square_b$ on  $L^2\left(L(k;\ell_1,\ldots,\ell_n)\right)$ that are at most $\lambda$.
\begin{theorem}\label{thm:oneoverk}
Given a lens space $L\left(k;\ell_1,\ldots,\ell_n\right)$, we have %we denote the eigenvalue counting function for $\square_b$ on the lens space by $N_L(\lambda)$, and the eigenvalue counting function for $S^{2n-1}$ by $N(\lambda)$. We have
\[\lim_{\lambda \rightarrow\infty}\frac{N_L(\lambda)}{N(\lambda)}=\frac{1}{k}.\]
\end{theorem}
Note that since the volume of the lens space is scaled by $k$, it is not surprising that the eigenvalue counting functions scale similarly. Furthermore, combining this result with the explicit calculations in \cite{REU2020Weyl}, yields the following analog of Weyl's law.

%From \cite{REU2020Weyl}, {\color{red} should we include this theorem from REU 2020 in the intro?} {\color{blue} colin - imo no} we obtain the following analog of Weyl's law.

\begin{corollary}
%Denoting the eigenvalue counting function for the lense space $L(k; \ell_1, \ldots, \ell_n)$ by $N_L(\lambda)$, 
We have
\[\lim_{n\to\infty} \frac{N_L\left(\lambda\right)}{\lambda^n} = u_n \frac{ \operatorname{vol}\left(S^{2n-1}\right)}{k} = u_n \operatorname{vol}\left(L\left(k;\ell_1,\ldots,\ell_n\right)\right),\]
where $u_n$ is a universal constant depending only on $n$, given by \[u_n = \frac{n - 1}{n \left(2\pi\right)^n \Gamma\left(n + 1\right)} \int_{-\infty}^{\infty} \left(\frac{x}{\sinh x}\right)^n e^{-(n - 2)x}\,dx.\]
\end{corollary}

As a consequence of Theorem \ref{thm:oneoverk}, the following corollary is immediate.

\begin{corollary} \label{cor:same_k}
If the lens spaces $L\left(k;\ell_1,  \ldots, \ell_n\right)$ and $L\left(k'; \ell_1', \ldots, \ell_n'\right)$ are CR isospectral, then $k = k'$.
\end{corollary}

With Corollary \ref{cor:same_k}, we obtain the following result on isospectral quotients of $S^3$.
\begin{theorem}
\label{thm:isospec}
Let $k$ be an odd prime number. Let $L\left(k; \ell_1, \ell_2\right)$ and $L\left(k; \ell_1', \ell_2'\right)$ be the lens spaces generated by the groups $G$ and $G'$, respectively.  The following are equivalent.
\begin{enumerate}
\item $L\left(k; \ell_1, \ell_2\right)$ and $L\left(k; \ell_1', \ell_2'\right)$ are CR isometric.
\item $L\left(k; \ell_1, \ell_2\right)$ and $L\left(k; \ell_1', \ell_2'\right)$ are CR isospectral.
\item $\dim \calH_{p,q}^G = \dim \calH_{p,q}^{G'}$ for all $p,q\geq 0$.
\item There exists an integer $a$ and a permutation $\sigma$ such that $\left(\ell_1',\ell_2'\right) \equiv \left(a\ell_{\sigma\left(1\right)},a\ell_{\sigma\left(2\right)}\right)\mod k.$
\end{enumerate}
\end{theorem}

%In our proof of Theorem \ref{thm:isospec}, we will show $(1)$ implies $(2)$, $(2)$ implies $(3)$, $(3)$ implies $(4)$, and $(4)$ implies $(1)$. 

\section{Diophantine Equations and Proof of Theorem \ref{thm:oneoverk}}

We begin our proof of Theorem \ref{thm:oneoverk} with a series of technical lemmas.

\begin{lemma}\label{lem:lim=0}
For $p,q \in \mathbb{Z}$, $p,q\geq 0$, define
\[a_{p,q}=\binom{
p+n-2}{n-2}\binom{q+n-2}{n-2} \text{ and } b_{p,q}=\left(\frac{p+q}{n-1}+1\right)\binom{p+n-2}{n-2}\binom{q+n-2}{n-2}.\]
The following equality holds,
\[\lim_{\lambda \rightarrow \infty} \frac{\sum_{p=0}^{\left\lfloor \lambda - n + 1 \right\rfloor}\sum_{q=1}^{\left\lfloor \frac{\lambda}{p+n-1} \right\rfloor}a_{p,q}}{\sum_{p=0}^{\left\lfloor \lambda - n + 1 \right\rfloor}\sum_{q=1}^{\left\lfloor \frac{\lambda}{p+n-1} \right\rfloor}b_{p,q}}=0.\]
\end{lemma}

%{\color{red}this lemma should follow essentially from the following: if $x_n$ is a positive sequence bounded away from zero, and $y_n$ is a sequence diverging to $\infty$, then $\left(\sum x_n y_n\right)/\sum x_n$ also diverges to $\infty$.}

\begin{proof}
The claim is intuitive due to the fact that $b_{p,q} = \left(\frac{p+q}{n-1}+1\right) a_{p,q}$. Let $A_m = \sum_{p=0}^{ m - n + 1 }\sum_{q=1}^{\left\lfloor \frac{m}{p+n-1} \right\rfloor}a_{p,q}$ and $B_m = \sum_{p=0}^{ m - n + 1 }\sum_{q=1}^{\left\lfloor \frac{m}{p+n-1} \right\rfloor}b_{p,q}$. 
Note that it suffices to show $B_m/A_m\to\infty$ as $m\to\infty$. Fix $R>0$, and note that $B_m/A_m > R$ if and only if
\[C_m =  \sum_{p=0}^{m- n + 1} \sum_{q=1}^{\left\lfloor \frac{m}{p + n - 1}\right\rfloor} \left(\frac{p + q}{n-1} + 1 - R \right) a_{p,q} > 0.\]
In particular, there is a negative part $N_R$ of the above sum contributed by indices where $\frac{p + q}{n-1}+ 1 - R < 0$. But for large enough $m$, this negative part is always dominated by the positive part. By taking $m \geq M + n - 1$ where $M$ is such that $\frac{M+1}{n-1} + 1 - R + N_R > 0$, the term given by $p = m - n + 1$ and $q = 1$  contributes enough to make $C_m$ positive.
\end{proof}

\subsection{Lower and Upper Bounds}

In this section we establish a lower and upper bound that will be used to control the number of solutions   to the Diophantine system in Proposition \ref{prop:system}.
We state first the lower bound and then the analogous upper bound.

%{\color{red}maybe it would be better to take $q = 0 $ in the $\frac{n-2}{q + n- 2 }$ terms to make it look nicer. that extra factor is not useful anyway.}
\begin{lemma}\label{lem:lower_bound}
If $N,m,d \in \mathbb{Z}_{\geq 0}$, $m,d > 0$, then
\[
    \sum_{r=0}^N\binom{r+n-3}{n-3}\sum_{j=0}^{\left\lfloor \frac{1}{m}(N-r+1)\right\rfloor - 1}\left\lfloor \frac{1}{d}(jm+1)\right \rfloor \geq \frac{1}{md}\sum_{q=0}^{N}\left(\frac{q+n-1}{n-1}-d-\frac{3}{2}m- \left( d + \frac{3}{2} m \right) \frac{n-2}{q + n - 2}\right) \binom{q + n - 2}{n-2},
\]
where in the case that the upper index of a summation is $-1$, we take it to be equal to zero.
\end{lemma}
\begin{lemma}\label{lem:upper_bound}
If $N,m,d \in \Z_{\geq 0}$, $m, d > 0$, then
\[  \sum_{r=0}^N\binom{r+n-3}{n-3}\sum_{j=1}^{\left\lceil \frac{1}{m}(N-r+1)\right\rceil}\left\lceil \frac{jm}{d}\right\rceil 
  \leq \frac{1}{md}\sum_{q=0}^N\left( \frac{q+n-1}{n-1}+d + \frac{3}{2}m + (m^2+md)\frac{n-2}{q+n-2}\right)\binom{q+n-2}{n-2}.
\]
\end{lemma}
We prove only the lower bound as the proof for the upper bound follows similarly.
\begin{proof}
First, note that for any $M \geq -1$,
\[\sum_{j=0}^M\left\lfloor \frac{jm+1}{d} \right\rfloor \geq \sum_{j=0}^M\frac{1}{d}\left(jm-d+1\right)
=\frac{1}{d}(M+1)\left(\frac{mM}{2}+1-d\right).\]
Thus, it suffices to give the claimed lower bound for the sum
\[\sum_{r=0}^N\binom{r+n-3}{n-3}\frac{1}{d}\left\lfloor \frac{N-r+1}{m}\right\rfloor \left( \frac{m\left(\lfloor \frac{N-r+1}{m}\rfloor - 1 \right)}{2} + 1 - d \right).\]
We see that for $r \leq N$,
\begin{align*}
\left\lfloor \frac{N-r+1}{m}\right\rfloor\left(\frac{m\left(\left\lfloor \frac{N-r+1}{m}\right\rfloor -1\right)}{2}+1-d\right) 
&\geq \frac{1}{m} \left(N - r + 1 - m\right) \left(\frac{N - r + 1}{2} - m + 1 - d\right)\\
&=\frac{1}{m} \left( \frac{\left(N - r + 2\right)^2}{2} - \left(d + \frac{3}{2}m\right)\left(N- r + 2 \right) - \left(1 + m\right)\left(\frac{1}{2} - m - d \right)\right)\\
&\geq \frac{1}{m} \left( \frac{\left(N - r + 2\right)^2}{2} - \left(d + \frac{3}{2}m\right)\left(N- r + 1 \right) - \left(d + \frac{3}{2}m\right) \right)\\
&\geq \frac{1}{m} \left( \binom{N - r + 2}{2} - \left(d + \frac{3}{2}m\right)\left(N- r + 1 \right) - \left(d + \frac{3}{2}m\right) \right).
\end{align*}
Now, the combinatorial identities,
\[\sum_{q=0}^N\binom{q+n-1}{n-1}=\binom{N+n}{n}\quad \text{and} \quad \sum_{m=0}^n\binom{m}{j}\binom{n-m}{k-j}=\binom{n+1}{k+1}\text{ for }0 \leq j \leq k \leq n\]
imply %{\color{red} this step is a little too quick for me}
\[\sum_{r=0}^N\binom{r+n-3}{n-3}\binom{N-r+2}{2}=\sum_{q=0}^N\binom{q+n-1}{n-1}\text{ and }\sum_{r=0}^N\binom{r+n-3}{n-3}\binom{N-r+1}{1}=\sum_{q=0}^N\binom{q+n-2}{n-2}.\]
Applying these identities and the lower bound above, the claim follows,
\begin{align*}
\sum_{r=0}^N\binom{r+n-3}{n-3}\frac{1}{d}\left\lfloor \frac{N-r+1}{m}\right\rfloor &\left( \frac{m\left(\lfloor \frac{N-r+1}{m}\rfloor - 1 \right)}{2} + 1 - d \right)
\\
&\geq\sum_{r=0}^N\binom{r+n-3}{n-3}\frac{1}{md}\left( \binom{N - r + 2}{2} - \left(d + \frac{3}{2}m\right)\left(N- r + 1 \right) - \left(d + \frac{3}{2}m\right) \right)\\
&= \frac{1}{md} \sum_{q=0}^N \left( \frac{q + n - 1}{n - 1 } - d - \frac{3}{2} m - \left( d + \frac{3}{2} m \right) \frac{n-2}{q + n - 2}\right) \binom{q + n - 2}{n-2}.
\end{align*}
\end{proof}

\subsection{Bounding the Number of Solutions to the Diophantine System}
Given a lens space $L = L\left(k; \ell_1,\ldots, \ell_n\right)$, with eigenvalue counting function $N_L(\lambda)$, note
\[N_L\left(2\lambda\right) = \sum_{p=0}^{\left\lfloor \lambda - n + 1 \right\rfloor}\sum_{q=1}^{\left\lfloor \frac{\lambda}{p+n-1}\right\rfloor }\dim\mathcal{H}_{p,q}^G\]
and
\[N\left(2\lambda\right) = \sum_{p=0}^{\left\lfloor \lambda - n + 1 \right\rfloor}\sum_{q=1}^{\left\lfloor \frac{\lambda}{p+n-1}\right\rfloor }\dim\mathcal{H}_{p,q} = \sum_{p=0}^{\left\lfloor \lambda - n + 1 \right\rfloor}\sum_{q=1}^{\left\lfloor \frac{\lambda}{p+n-1}\right\rfloor } \left(\frac{p+q}{n-1} + 1\right)\binom{p+n-2}{n-2}\binom{q+n-2}{n-2},\]
where $N\left(\lambda\right)$ is the eigenvalue counting function the Kohn Laplacian on $S^{2n-1}$. These formulas are due to the fact that the eigenvalues corresponding to $\mathcal{H}_{p,q}$ and $\mathcal{H}_{p,q}^G$ are $2q \left(p + n - 1 \right)$.

Recall from Proposition \ref{prop:system}, $\dim\mathcal{H}_{p,q}^G$ is equal to the number of solutions $\left(\alpha,\beta\right)$ to the Diophantine system
\[ %\label{eq:important!}
\sum_{j=1}^{n}\ell_j\left(\alpha_j-\beta_j\right) \equiv 0 \tn{ mod } k,\]
where    $\alpha_j,\beta_j \geq 0$, $\alpha_1\beta_1 = 0$, and $\left|\alpha\right| = p$ and $\left|\beta\right| = q$.

Since our goal is to study $N_L\left(2\lambda\right)$, we can replace the conditions $\left|\alpha\right| = p$ and  $\left|\beta\right| = q$ with  $0 < \left|\beta\right|\left(\left|\alpha\right|+n-1\right) \leq \lambda$. That is, for fixed $\lambda > 0$, we study the number of solutions $\left(\alpha,\beta\right)$ to the Diophantine equation
 \begin{align}
 \sum_{j=1}^{n}\ell_j\left(\alpha_j-\beta_j\right) \equiv 0 \tn{ mod } k, 
 \label{eq:important!} 
 \end{align}
 where $\alpha_j,\beta_j \geq 0$, $\alpha_1 \beta_1 = 0$, and $     0 < \left|\beta\right|\left(\left|\alpha\right|+n-1\right) \leq \lambda $.
 
In particular, we establish upper and lower bounds for the number of solutions to the system (\ref{eq:important!}) in the following two lemmas, where they are distinguished by fixing $\alpha_1 = 0$ or $\beta_1 = 0$. We will make use of Lemma \ref{lem:lower_bound} and Lemma \ref{lem:upper_bound} as well as the following fact,
\begin{proposition*}
Given $a,b,c,d \in \N$, with $a \leq b$, the number of solutions to the system
\begin{align*}
    a &\leq x \leq b\\
    x &\equiv c \tn{ mod }d
\end{align*}
is either $\left\lfloor \frac{b-a+1}{d} \right\rfloor$ or $\left\lceil \frac{b-a+1}{d}\right \rceil$.
\end{proposition*}
\begin{lemma} \label{lem:alpha1=0}
For $\lambda>n$, when $\alpha_1=0$, the number of solutions to the system (\ref{eq:important!}) is greater than 
\[\frac{1}{k}\sum_{p=0}^{\left\lfloor\lambda - n + 1\right\rfloor}\binom{p+n-2}{n-2}\sum_{q=0}^{\left\lfloor \frac{\lambda}{p+n-1}\right\rfloor}\left(\frac{q+n-1}{n-1}-d-\frac{3}{2}m - \left(d + \frac{3}{2} m\right) \frac{n-2}{q + n - 2}\right)\binom{q+n-2}{n-2}\]
and less than 
\[\frac{1}{k}\sum_{p=0}^{\left\lfloor \lambda - n + 1 \right\rfloor}\binom{p+n-2}{n-2}\sum_{q=0}^{\left\lfloor \frac \lambda {p+n-1} \right\rfloor}\left( \frac{q+n-1}{n-1}+ d + \frac{3}{2}m + \left(m^2+md\right)\frac{n-2}{q+n-2}\right)\binom{q+n-2}{n-2}\]
where $m = \gcd \left(k,\ell_1-\ell_2\right)$ and $d = k/m$.
\end{lemma}
\begin{lemma} \label{lem:beta1=0}
For $\lambda > n$, when $\beta_1=0$, the number of solutions to the system (\ref{eq:important!}) is greater than 
\[\frac{1}{k}\sum_{q=1}^{\left\lfloor \frac{\lambda}{n-1}\right\rfloor}\binom{q+n-2}{n-2}\sum_{p=0}^{\left\lfloor \frac{\lambda}{q}-n+1  \right\rfloor}\left(\frac{p+n-1}{n-1}-d-\frac{3}{2}m - \left(d + \frac{3}{2}m\right) \frac{n-2}{p + n - 2}\right)\binom{p+n-2}{n-2}\]
and less than 
\[\frac{1}{k}\sum_{q=1}^{\left\lfloor \frac{\lambda}{n-1} \right\rfloor}\binom{q+n-2}{n-2}\sum_{p=0}^{\left\lfloor\frac{\lambda}{q}-n+1\right\rfloor}\left( \frac{p+n-1}{n-1} + d + \frac{3}{2}m + \left(m^2+md\right)\frac{n-2}{p+n-2} \right)\binom{p+n-2}{n-2}\]
where $m = \gcd \left(k,\ell_1-\ell_2\right)$ and $d = k/m$.
\end{lemma}
We prove only Lemma \ref{lem:alpha1=0}, as the proof of Lemma \ref{lem:beta1=0} follows the same argument.
\begin{proof}
The outline of the argument is as follows. Recall that we want to estimate the number of $\left(\alpha,\beta\right)$ solving the Diophantine system, (\ref{eq:important!}). To do this, we first note that by stars and bars, we can find a number of candidates for $\alpha$. Thus, by fixing some $\alpha'$, it suffices to find the number of $\beta$ such that $\left(\alpha',\beta\right)$ is a solution to (\ref{eq:important!}). To reduce the problem further, by another application of stars and bars, we can find a number of values of $\beta_3,\ldots,\beta_n$ that may contribute to solutions of (\ref{eq:important!}). So for fixed $\alpha',\beta_3,\ldots,\beta_n$, by using the above proposition, we estimate the number of $\beta_1$ so that $\left(\alpha',\beta_1,\beta_2,\beta_3,\ldots,\beta_n\right)$ solve (\ref{eq:important!}). Note that $\beta_2$ is determined by $\beta_1,\beta_3,\ldots,\beta_n$.

We now proceed with the details. For $p \in \Z_{\geq 0}$, there are $\binom{p+n-2}{n-2}$ values of $\alpha$ where $\alpha_1=0$ and $\left|\alpha\right|=p$. Similarly, for $r \in \Z_{\geq 0}$, there are $\binom{r+n-3}{n-3}$ values for $\beta_3,\ldots,\beta_n$ such that $\sum_{j=3}^n\beta_j=r$. Now we would like to compute bounds on the number of $\beta_1$ that yield a solution to (\ref{eq:important!}) along with $\alpha,\beta_3,\ldots,\beta_n$ for a fixed value of $\left|\alpha\right|=p$ and $\sum_{j=3}^n\beta_j=r$, and fixed $0 \leq p \leq \left\lfloor\lambda-n+1\right\rfloor$. Note that from this information, $q = \left|\beta\right|$ is not fixed and it can take different values as $\beta_1$ changes. However, we may write $\beta_2 = q - r - \beta_1$ and therefore,
%Given $\alpha$ and $\beta_3,\ldots, \beta_n$ satisfying this, we first note that the conditions on $\beta_1$ are further restricted as we must have $r \leq q \leq \left\lfloor\frac{\lambda}{p+n-1}\right\rfloor$, and $0 \leq \beta_1 \leq q - r$, where $q=\left|\beta\right|$.
\[\left(\ell_1-\ell_2\right)\beta_1\equiv \sum_{j=2}^n\ell_j\alpha_j - \ell_2\left(q-r\right) - \sum_{j=3}^n\ell_j\beta_j \tn{ mod } k.\]
Since $m = \tn{gcd}\left(k,\ell_1-\ell_2\right)$ and $dm = k$, there exists a $d'$ coprime to $d$ so that $d'm = \ell_1 - \ell_2$. It follows that,
\[d'm\beta_1 \equiv \sum_{j=2}^n\ell_j\alpha_j - \ell_2\left(q-r\right) - \sum_{j=3}^n\ell_j\beta_j \tn{ mod } md.\]
The existence of $\beta_1$ that satisfies this equation is equivalent to $m$ dividing $\sum_{j=2}^n \ell_j\alpha_j -\ell_2\left(q-r\right)-\sum_{j=3}^n \ell_j \beta_j$ as $d$ and $d'$ are coprime. Furthermore, this division condition is equivalent to 
\begin{equation}\label{eq:q_congruence}
   q \equiv \ell_2^{-1}\left(\ell_2r+\sum_{j=2}^n\ell_j\alpha_j-\sum_{j=3}^{n}\ell_j\beta_j\right) \tn{ mod } m.
\end{equation}
Since $q$ is subject to the conditions $r \leq q \leq \left\lfloor\frac{\lambda}{p+n-1}\right\rfloor$, by the above proposition, the number of solutions to (\ref{eq:q_congruence}) is bounded by
\[\left\lfloor \frac{1}{m}\left( \left\lfloor \frac{\lambda}{p+n-1}\right\rfloor - r + 1)\right) \right\rfloor\text{ and } \left\lceil \frac{1}{m}\left( \left\lfloor \frac{\lambda}{p+n-1}\right\rfloor - r + 1)\right) \right\rceil.\]
Note that if there is no $\beta_1$ that solves (\ref{eq:important!}) given $\alpha,\beta_3,\ldots,\beta_n$, then the lower bound is zero with no harm. This case corresponds to when the sum $\sum_{j=0}^{\left\lfloor \frac{1}{m} \left(N - r + 1 \right)\right\rfloor - 1} \left\lfloor \frac{1}{d} \left(jm + 1 \right)\right\rfloor$ has its upper index equal to $-1$ in Lemma \ref{lem:lower_bound}. In particular, we continue with our analysis by assuming that there is such a $\beta_1$, and therefore $q$, satisfying (\ref{eq:q_congruence}). 
Note that if (\ref{eq:q_congruence}) is satisfied, we can divide the equation by $m$ to obtain
\begin{equation}\label{eq:b1_congruence}
    \beta_1 \equiv \left(d'\right)^{-1}\frac{1}{m}\left(\sum_{j=2}^n\ell_j\alpha_j-\ell_2\left(q-r\right)-\sum_{j=3}^n\ell_j\beta_j\right) \tn{ mod } d.
\end{equation}
For a fixed value of $q$ which solves (\ref{eq:q_congruence}), since $0 \leq \beta_1 \leq q-r$, we can bound the number of solutions to (\ref{eq:b1_congruence}) by
\[\left\lfloor \frac{1}{d}(q-r+1) \right\rfloor \text{ and }\left\lceil \frac{1}{d}(q-r+1) \right\rceil.\]
First we look at the lower bound.
The values for $q$ which solve (\ref{eq:q_congruence}) are of the form $r+c + jm$, where $0 \leq c < m$ is some integer and $j$ is also an integer ranging from $0$ to the number of solutions of (\ref{eq:q_congruence}) minus one. Note that for all $j$, 
\[\left\lfloor \frac1d((r+c+jm) - r + 1) \right\rfloor \geq \left\lfloor \frac1d(jm+1) \right\rfloor,\]
and therefore a lower bound on the number of solutions to both (\ref{eq:q_congruence}) and (\ref{eq:b1_congruence}) is
\[\sum_{j=0}^{\left\lfloor \frac{1}{m}\left( \left\lfloor \frac{\lambda}{p+n-1}\right\rfloor - r + 1\right) \right\rfloor-1}\left\lfloor\frac{1}{d}(jm+1)\right\rfloor.\]
So a lower bound on the number of solutions to (\ref{eq:important!}) when $\alpha_1=0$ is
\[\sum_{p=0}^{\lfloor\lambda - n + 1\rfloor}\binom{p+n-2}{n-2}\sum_{r=0}^{\left\lfloor \frac{\lambda}{p+n-1}\right\rfloor}\binom{r+n-3}{n-3}\sum_{j=0}^{\left\lfloor \frac{1}{m}\left( \left\lfloor \frac{\lambda}{p+n-1}\right\rfloor - r + 1\right) \right\rfloor-1}\left\lfloor\frac{1}{d}(jm+1)\right\rfloor.\]
By Lemma \ref{lem:lower_bound}, the first part of the lemma follows.

Now we examine the upper bound. Note that for all $j$,
\[\left\lceil \frac1d((r+c+jm)-r+1)\right\rceil \leq \left\lceil \frac 1d(j+1)m \right\rceil.\]
So an upper bound on the number of solutions to the system when $\alpha_1 = 0$ is
\[\sum_{p=0}^{\left\lfloor \lambda - n + 1 \right\rfloor}\binom{p+n-2}{p}\sum_{r=0}^{\left\lfloor \frac \lambda {p+n-1} \right\rfloor}\binom{r+n-3}{n-3}\sum_{j=1}^{\left\lceil \frac{1}{m}(\lfloor \frac \lambda {p+n-1} \rfloor-r+1) \right\rceil}\left\lceil \frac{jm}{d}\right\rceil.\]
By Lemma \ref{lem:upper_bound}, the second part of the lemma follows.
\end{proof}

\subsection{Proof of Theorem \ref{thm:oneoverk}}
We now use Lemma \ref{lem:alpha1=0} and Lemma \ref{lem:beta1=0} and our reformulation of $N_L\left(2\lambda\right)$ as the number of solutions to equation (\ref{eq:important!}) to prove Theorem \ref{thm:oneoverk}.

%\begin{theorem}
%Setting $N_L(\lambda)$ as the eigenvalue counting function for $\square_b$ on a lens space $L(k; l_1, \ldots, l_n)$, we have that $$\lim_{\lambda \rightarrow \infty} \frac{N_L(\lambda)}{N(\lambda)}=\frac{1}{k}.$$
%\end{theorem}

\begin{proof}
To get a lower bound on sum $N_L(2\lambda)$, we note that we have three parts.
First, from Lemma \ref{lem:alpha1=0}, the part where $\alpha_1=0$:
\begin{align*}
\frac{1}{k}\sum_{p=0}^{\lfloor \lambda - n + 1 \rfloor}\sum_{q=0}^{\left\lfloor \frac{\lambda}{p+n-1}\right\rfloor}&\left(\frac{q+n-1}{n-1}-d-\frac{3}{2}m - \left(d + \frac{3}{2} m \right) \frac{n-2}{q + n - 2} \right)\binom{p+n-2}{n-2}\binom{q+n-2}{n-2}\\
\hspace{-20mm}=& \frac{1}{k}\sum_{p=0}^{\lfloor \lambda - n + 1 \rfloor}\sum_{q=1}^{\left\lfloor \frac{\lambda}{p+n-1}\right\rfloor}\left(\frac{q+n-1}{n-1}-d-\frac{3}{2}m - \left(d + \frac{3}{2} m \right) \frac{n-2}{q + n - 2} \right)\binom{p+n-2}{n-2}\binom{q+n-2}{n-2}\\
&+\frac{1}{k}\sum_{p=0}^{\lfloor\lambda-n+1\rfloor}\left(1-2d-3m \right)\binom{p+n-2}{n-2}.
\end{align*}

Next, from Lemma \ref{lem:beta1=0}, the part where $\beta_1=0$:
\[\frac{1}{k}\sum_{p=0}^{\left\lfloor \lambda - n + 1 \right\rfloor}\sum_{q=1}^{\left\lfloor \frac{\lambda}{p+n-1} \right\rfloor}\left( \frac{p+n-1}{n-1} - d -\frac{3}{2}m - \left(d + \frac{3}{2} m \right) \frac{n-2}{p + n - 2}\right)\binom{p+n-2}{n-2}\binom{q+n-2}{n-2}.\]
Note that here we swapped the order of summation from the original statement of Lemma \ref{lem:beta1=0}. 
Finally, we have the part to be subtracted off which is given by $\alpha_1=\beta_1=0$. This part is bounded above by a formula from stars and bars:
\[\sum_{p=0}^{\left\lfloor \lambda - n + 1 \right\rfloor}\sum_{q=1}^{\left\lfloor \frac{\lambda}{p+n-1} \right\rfloor}\binom{p+n-2}{n-2}\binom{q+n-2}{n-2}.\]
Putting this together, a lower bound of $N_L \left(2\lambda\right)$ is
\[
\frac{1}{k}\sum_{p=0}^{\left\lfloor \lambda - n + 1 \right\rfloor}\sum_{q=1}^{\left\lfloor \frac{\lambda}{p+n-1} \right\rfloor} \left(\frac{p + q}{n - 1} + A +\frac{B}{q + n - 2} + \frac{C}{p + n - 2}\right)\binom{p+n-2}{n-2}\binom{q+n-2}{n-2} + \frac{1}{k}\sum_{p=0}^{\lfloor\lambda-n+1\rfloor}D\binom{p+n-2}{n-2},\]
where $A,B,C,D$ are constants dependent only on $m,d$ and $n$. By Lemma \ref{lem:lim=0}, we know
\[\lim_{\lambda \rightarrow \infty} \frac{\sum_{p=0}^{\lfloor \lambda - n + 1 \rfloor}\sum_{q=1}^{\left\lfloor \frac{\lambda}{p+n-1} \right\rfloor}\binom{p+n-2}{n-2}\binom{q+n-2}{n-2}}{\sum_{p=0}^{\lfloor \lambda - n + 1 \rfloor}\sum_{q=1}^{\left\lfloor \frac{\lambda}{p+n-1} \right\rfloor}\left(\frac{p+q}{n-1}+1\right)\binom{p+n-2}{n-2}\binom{q+n-2}{n-2}}=0.\]
Therefore, $\lim_{\lambda\to\infty} \frac{N_L\left(\lambda\right)}{N \left(\lambda\right)}\geq \frac{1}{k}$. The upper bound follows similarly, except there is no need to consider the part where $\alpha_1 = \beta_1 = 0$.
\end{proof}
\begin{remark}
This proof shows furthermore that $N_L\left(2\lambda\right)$ has asymptotic behavior, \[u_n \operatorname{vol} \left(L\left(k;\ell_1,\ldots,\ell_n\right)\right) + O \left(\sum_{p=0}^{\left\lfloor \lambda - n + 1 \right\rfloor} \sum_{q = 1}^{\left\lfloor \frac{\lambda}{p + n - 1 } \right\rfloor} \binom{ p + n - 2}{n - 2} \binom{q + n - 2}{n - 2}\right).\]By some %informal 
calculations, we conjecture that this remainder term can be bounded above by $O \left(\lambda^{n - 1} \log \lambda\right)$ but not $O \left(\lambda^{n-1}\right)$. If true, this would be similar to the remainder term
%Strichartz's remainder term estimate 
on compact Heisenberg manifolds in \cite{strichartz2015}, except the fact that the authors speculates in the paper that the estimate can be improved to $O\left(\lambda^{n}\right)$.
\end{remark}

\section{Isospectral lens spaces and Proof of Theorem \ref{thm:isospec}}
In this section, we provide more details and a proof of Theorem \ref{thm:isospec}. For convenience, we restate it here.

\noindent\textbf{Theorem \ref{thm:isospec}.}
Let $k$ be an odd prime number. Let $L(k; \ell_1, \ell_2)$ and $L(k; \ell_1', \ell_2')$ be the lens spaces generated by the groups $G$ and $G'$, respectively.  The following are equivalent.
\begin{enumerate}
\item $L(k; \ell_1, \ell_2)$ and $L(k; \ell_1', \ell_2')$ are CR isometric.
\item $L(k; \ell_1, \ell_2)$ and $L(k; \ell_1', \ell_2')$ are CR isospectral.
\item $\dim \calH_{p,q}^G = \dim \calH_{p,q}^{G'}$ for all $p,q\geq 0$.
\item There exists an integer $a$ and a permutation $\sigma$ such that $\left(\ell_1',\ell_2'\right) \equiv \left(a\ell_{\sigma\left(1\right)},a\ell_{\sigma\left(2\right)}\right)\mod k$.
\end{enumerate}

\subsection{(2) implies (3)}

%{\color{red}i think we can be more precise than the red below. perhaps say the only places where we use k being prime is here and odd here.}

Of the four implications we show, this is the only one that requires $k$ to be an odd prime. In particular, we need this assumption after Lemma \ref{Lem:dimension}. Fix a lens space $L(k;\ell_1,\ell_2)$ and let $G$ be the corresponding group. Let $d= \gcd(k, \ell_1- \ell_2)$. We begin with some observations about $\dim\calH^G_{p,q}$.
\begin{lemma}
\label{lem:symmetric}
$\dim\calH^G_{p,q} =  \dim\calH^G_{q,p}$ for all $p,q$.
\end{lemma}
This follows from Proposition \ref{prop:system}, observing that $(\alpha,\beta)\leftrightarrow(\beta,\alpha)$ is a one-to-one correspondence between solutions to the system for $\calH^G_{p,q}$ and solutions to the system for $\calH^G_{q,p}$.

%\begin{proof}
%\end{proof}
\begin{lemma}
\label{lem:doesnotdivide} 
If $d$ does not divide $p-q$, then $\dim\calH_{p,q}^G = 0$.
\end{lemma}
\begin{proof} From Proposition~\ref{prop:system}, we know $\dim\calH_{p,q}^G$ is the number of solutions to \[\ell_1(\alpha_1 - \beta_1) + \ell_2(\alpha_2 - \beta_2) \equiv 0\mod k\]
such that $\alpha_1 \beta_1 = 0$, and $\alpha_1 + \alpha_2 = p$, $\beta_1+\beta_2 =q$. Suppose for contradiction $d$ does not divide $p-q$, but a solution exists. We then have
\begin{align*}
\ell_1\left(\alpha_1 - \beta_1\right) + \ell_2\left(\left(p - \alpha_1\right) - \left(q-\beta_1\right)\right) &\equiv 0 \mod k\\
\left(\ell_1 - \ell_2\right)\left(\alpha_1 - \beta_1\right)  &\equiv - \ell_2\left(p-q\right) \mod k.
\end{align*}
Note that $d$ is coprime to $\ell_2$, since $\ell_2$ is coprime to $k$ and $d$ divides $k$. Since $d$ does not divide $p-q$, we see that $d$ does not divide the right hand side of the equation. However, $d$ divides $\ell_1 - \ell_2$, which is a contradiction, meaning no such solution exists.
\end{proof}
We now relate the dimensions of invariant subpaces.

\begin{lemma}%[Dimension Lemma]
If $d$ divides $p-q$, then \[\dim \calH^G_{p+k,q} =\dim \calH^G_{p,q+k} = d+ \dim\calH^G_{p,q}. \]
\end{lemma}
\begin{proof} From Proposition~\ref{prop:system}, $\dim\calH_{p,q}^G$ is the number of solutions to 
\[\ell_1\left(\alpha_1 - \beta_1\right) + \ell_2\left(\alpha_2 - \beta_2\right) \equiv 0\mod k\]
such that $\alpha_1 + \alpha_2 = p$, $\beta_1+\beta_2 =q$, and either $\alpha_1 = 0$ or $\beta_1 = 0$. In the $\alpha_1 = 0$ case, this reduces to 
\begin{equation}
\label{eq:alpha}
\ell_1\left(-\beta_1\right) + \ell_2\left(p - q +\beta_1\right) \equiv 0\mod k\quad\text{ where } 0\leq \beta_1 \leq q.
\end{equation}
Similarly, in the $\beta_1 = 0$ case, we obtain
\begin{equation}
\label{eq:beta}
\ell_1\left(\alpha_1\right) + \ell_2\left(p - q -\alpha_1\right) \equiv 0\mod k\quad\text{ where } 0\leq \alpha_1 \leq p.
\end{equation}
By inclusion-exclusion, $\dim\calH_{p,q}^G$ is the number of solutions to (\ref{eq:alpha}) plus the number of solutions to (\ref{eq:beta}) minus the number of solutions where $\alpha_1 = \beta_1 = 0$. Note that in the $\alpha_1 = \beta_1 = 0$ case, the Diophantine system reduces to $\ell_1\left(p-q\right) \equiv 0\mod k$, which has one solution if $k$ divides $p-q$, and zero solutions otherwise. For ease of notation, define the following indicator function for divisibility
\[\divides\left(k, a\right) = \begin{cases}
1 & k \text{ divides } a\\
0 & \text{otherwise}
\end{cases}.\]

Let $m_{p,q}$ denote the number of solutions to (\ref{eq:alpha}) and $n_{p,q}$ denote the number of solutions to (\ref{eq:beta}). Using this notation, the above inclusion-exclusion argument can be written as \[\dim\calH_{p,q}^G = m_{p,q} + n_{p,q} - \divides\left(k, p-q\right).\]

Note that $m_{p+k,q}$ is the number of solutions to $\ell_1\left(-\beta_1\right) + \ell_2\left(p+k-q+\beta_1\right) \equiv 0\mod k$ for $0\leq \beta_1\leq q$. This equation is equivalent to $\ell_1\left(-\beta_1\right) + \ell_2\left(p-q+\beta_1\right) \equiv 0\mod k$, which implies $m_{p+k,q} = m_{p,q}$. By the same logic, we can see that $n_{p,q+k} = n_{p,q}$. 

We now show $m_{p,q+k} = m_{p,q}+d$ and $n_{p+k,q} = n_{p,q} + d$. Note that $m_{p,q+k} - m_{p,q}$ is the number of solutions to $\ell_1\left(-\beta_1\right) + \ell_2\left(p-q+\beta_1\right) \equiv 0\mod k$ where $q+1 \leq \beta_1\leq q+k$. Reducing modulo $k$, this is equal to the number of solutions where $1\leq \beta_1\leq k$. The equivalence may be rewritten as 
\[\left(\ell_1 - \ell_2\right)\beta_1 \equiv \ell_2\left(p-q\right)\mod k, \] where $1 \leq \beta_1 \leq k$.
Since $d$ divides $\ell_1 - \ell_2$ and $p-q$, we obtain 
\[\left(\frac{\ell_1-\ell_2}{d}\right)\beta_1 \equiv \ell_2\left(\frac{p-q}{d}\right)\mod{k/d}.\]
Since $d = \gcd(k,\ell_1-\ell_2)$, the above equivalence has a unique solution for $1\leq \beta_1\leq k/d$. Thus, there are precisely $d$ solutions in $1\leq \beta_1 \leq k$. Hence, $m_{p,q+k} = m_{p,q} + d$. Similarly, we see that $n_{p+k,q} = n_{p,q} + d$.

Combining these facts, we have 
\[
\dim\calH_{p+k,q}^G = m_{p+k,q} + n_{p+k,q} - \divides\left(k, p-q\right)
= m_{p,q} + n_{p,q} + d - \divides\left(k, p-q\right)
= \dim\calH_{p,q}^G +d\]
and
\[
\dim\calH_{p,q+k}^G = m_{p,q+k} + n_{p,q+k} - \divides\left(k, p-q\right)= m_{p,q}+d + n_{p,q} - 
\divides\left(k, p-q\right)= \dim\calH_{p,q}^G +d\]
completing the proof. 
\end{proof}

In what follows, it will be useful to distinguish between ``mod" as an equivalence relation and ``mod" as a binary operator. To that end, we will use continue to use ``$\mod k$" to denote the equivalence relation modulo $k$, and we will use ``$\%$" to denote the binary modulo operator. That is to say, $p\% k$ denotes the smallest non-negative integer equivalent to $p$ modulo $k$. For example, $11\%5 = 1$, $12\%4 = 0$, and $1\%7 = 1$. Using this notation, the previous lemma implies

\begin{lemma}[Dimension Lemma]\label{Lem:dimension}
%\label{cor:dimension}
If $d$ divides $p-q$, then 
\[\dim \calH^G_{p,q} = \dim\calH^G_{p\% k,q\% k} + d\left(\left\lfloor\frac{p}{k}\right\rfloor + \left\lfloor\frac{q}{k}\right\rfloor\right). \]
\end{lemma}

Now, let $L(k;\ell_1',\ell_2')$ be another lens space generated by the group $G'$, and let $d' = \gcd(k,\ell_1' - \ell_2')$.

\begin{proposition}\label{Propd=d'} Consider any two lens spaces $L\left(k;\ell_1,\ell_2\right)$ and $L\left(k;\ell_1',\ell_2'\right)$ where $d = 2$ and $d' = 4$ do not occur simultaneously. If the lens spaces are CR isospectral, then $d = d'$.
\label{prop:gcd}
\end{proposition}

\begin{remark}
We suspect that the above statement for lens spaces with $d = 2$ and $d' = 4$ is also true. This case seems to require a more subtle analysis than the growth rate arguments considered below. However, since in the rest of this section we assume $k$ is prime, this restriction on $d$ and $d'$ does not impact the rest of the paper. 
\end{remark}
\begin{proof} We show that if $d \neq d'$, then the spectra differ.

{\bf Case 1:} Assume $2 < d < d'$. Let $r$ be a prime to be determined such that $r \equiv 1 \mod d d'$. Let $\lambda = 4r$. It follows that the only $\left(p,q\right)$ so that $ 2 q \left( p + 1 \right) = \lambda$ are $\left(2r-1,1\right)$, $\left(r-1,2\right)$, $\left(1,r\right)$, and $\left(0,2r\right)$. Since $r\equiv 1 \mod d$ and $r\equiv 1 \mod d'$, by Lemma \ref{lem:doesnotdivide}
\begin{align*}
\operatorname{mult}_G \left(\lambda\right) 
&= \dim \mathcal{H}^G_{2r - 1,1} + \dim \mathcal{H}^G_{1,r}\\
\operatorname{mult}_{G'} \left(\lambda\right)
&= \dim \mathcal{H}^{G'}_{2r - 1,1} + \dim \mathcal{H}^{G'}_{1,r}.
\end{align*}
By the dimension lemma,
\begin{align*}
\operatorname{mult}_G \left(\lambda\right) 
&= \dim \mathcal{H}^G_{2r - 1\% k,1} + \dim \mathcal{H}^G_{1,r\% k} + d \left(\left\lfloor \frac{2r - 1}{k}\right\rfloor + \left\lfloor \frac{r}{k}\right\rfloor\right)\\
\operatorname{mult}_{G'} \left(\lambda\right)
&= \dim \mathcal{H}^{G'}_{2r - 1\% k,1} + \dim \mathcal{H}^{G'}_{1,r\% k}+ d' \left(\left\lfloor \frac{2r - 1}{k}\right\rfloor + \left\lfloor \frac{r}{k}\right\rfloor\right).
\end{align*}
Since $d < d'$, the multiplicity will differ for sufficiently large $r$.

{\bf Case 2:} Assume $2 = d < d'$ where $d' \neq 4$. Let $r$ be a prime to be determined such that $r \equiv 1 \mod d'$. Let $\lambda = 4r$. By Lemma \ref{lem:doesnotdivide}, we see that,
\begin{align*}
\operatorname{mult}_G \left(\lambda\right) 
&= \dim \mathcal{H}_{1,r}^G + \dim \mathcal{H}_{2r-1,1}^G + \dim \mathcal{H}_{0,2r}^G + \dim \mathcal{H}_{r-1,2}^G\\
\operatorname{mult}_{G'} \left(\lambda\right)
&= \dim \mathcal{H}_{1,r}^{G'} + \dim \mathcal{H}_{2r-1,1}^{G'}.
\end{align*}
By the dimension lemma,
\begin{align*}
\operatorname{mult}_{G} \left(\lambda\right) 
&= \dim \mathcal{H}_{1,r\% k}^G + \dim \mathcal{H}_{2r-1 \% k, 1}^G + \dim \mathcal{H}_{0,2r\% k}^G + \dim\mathcal{H}_{r - 1 \% k, 2}^G + 2 \left(\left\lfloor \frac{r}{k}\right\rfloor + \left\lfloor \frac{2r - 1}{k}\right\rfloor + \left\lfloor \frac{2r}{k}\right\rfloor + \left\lfloor \frac{r - 1}{k}\right\rfloor\right)\\
\operatorname{mult}_{G'} \left(\lambda\right)
&= \dim \mathcal{H}_{1,r\% k}^{G'} + \dim \mathcal{H}_{2r-1 \% k, 1}^{G'} + d' \left(\left\lfloor \frac{r}{k}\right\rfloor  + \left\lfloor \frac{2r - 1}{k}\right\rfloor\right).
\end{align*}
This implies $\operatorname{mult}_G \left(\lambda\right)$ has growth rate $\frac{12}{k}r$ while $\operatorname{mult}_{G'} \left(\lambda\right)$ has growth rate $\frac{3d'}{k}r$. Since the growth rates are different, for large enough $r$ the multiplicities will differ. 

{\bf Case 3:} Assume $1 = d < d'$. Let $r$ be any prime greater than $k$ so that $r\equiv 1 \mod d'$. Let $\lambda = 2r$. It follows that the only $\left(p,q\right)$ so that $2q \left(p+1\right) = \lambda$ are $\left(r-1,1\right)$ and $\left(0,r\right)$. Since $d'$ does not divide $r-2$ or $r$, $\calH_{r-1,1}^{G'}$ and $\calH_{0,r}^{G'}$ are both empty by Lemma \ref{lem:doesnotdivide}. This implies $\operatorname{mult}_{G'}\left(\lambda\right) = 0$. Since $r > k$, by the dimension lemma,
\[\operatorname{mult}_G\left(\lambda\right) = \dim\calH_{r-1,1}^{G} + \dim\calH_{0,r}^{G} = \dim\calH_{r-1\% k,1}^{G} + \dim\calH_{0,r\% k}^{G} + \lrfl{\frac {r-1}k} +\lrfl{\frac r k}> 0,\]
which implies the multiplicities differ. 
\end{proof}
Recall that in this section, (2) implies (3), we wish to show that if we have two CR isospectral lens spaces given by groups $G$ and $G'$, then $\dim \mathcal{H}_{p,q}^G = \dim \mathcal{H}_{p,q}^{G'}$ for all $p,q\geq 0$.

For each $p, q \geq 0$, define \[x_{p,q} = \dim \calH^G_{p,q}-\dim \calH^{G'}_{p,q}.\]
By the dimension lemma and Proposition \ref{prop:gcd}, we have
\[
x_{p,q}  =\dim\calH^G_{p\% k,q\% k} + d\left(\lrfl{\frac{p}{k}} + \lrfl{\frac{q}{k}}\right)-\left(\dim\calH^{G'}_{p\% k,q\% k} + d\left(\lrfl{\frac{p}{k}} + \lrfl{\frac{q}{k}}\right)\right)= x_{p\% k,q\% k}.\]
Thus, it suffices to show $x_{p,q} = 0$ for $0\leq p,q\leq k-1$. Define
\[X = \begin{pmatrix} 
x_{0,0} & x_{0,1} & \cdots & x_{0,k-1} \\
x_{1,0} & x_{1,1} & \cdots & x_{1,k-1} \\
\vdots &\vdots & \ddots & \vdots \\
x_{k-1,0} & x_{k-1,1} & \cdots & x_{k-1,k-1}\\
\end{pmatrix}.\]
By our observation above, $X = 0$ implies $\dim \calH^G_{p,q} = \dim \calH^{G'}_{p,q}$ for all $p,q\geq 0$. From Lemma \ref{lem:symmetric}, we see that $X$ is symmetric.

Assuming our two lens spaces are CR isospectral, for every $\lambda \in 2\N$, we have $$\sum_{2q(p+1) = \lambda}\dim \calH^G_{p,q} = \sum_{2q(p+1) = \lambda}\dim \calH^{G'}_{p,q}.$$ This implies that
\begin{align*}
0=\sum_{2q(p+1) = \lambda}\left(\dim \calH^G_{p,q} -\dim \calH^{G'}_{p,q}\right) 
=\sum_{2q(p+1) = \lambda}x_{p,q}
=\sum_{2q(p+1) = \lambda}x_{p\% k,q\% k}.
\end{align*}
For integers $a,b$ so that $0\leq a,b\leq k - 1$, define
\[c^\lambda_{a,b} = \#\left\{(p,q)\in \mathbb{Z}_{\geq 0}\times \mathbb{Z}_{\geq 0} :
p \equiv a\mod k,\,
q \equiv b\mod k,\,  
2q \left(p+1\right) = \lambda\right\}.\] 
Note that $c^\lambda_{a,b}$ is the number of times $x_{a,b}$ appears in the above sum. Therefore, we have \[\sum_{0 \leq a,b \leq k-1} c^\lambda_{a,b} x_{a,b} = 0. \]
For each $\lambda$, define the $k \times k$ matrix $C^{\lambda}$:
\[C^\lambda = \begin{pmatrix} 
c_{0,0}^\lambda & c_{0,1}^\lambda & \cdots & c_{0,k-1}^\lambda \\
c_{1,0}^\lambda & c_{1,1}^\lambda & \cdots & c_{1,k-1}^\lambda \\
\vdots & \vdots & \ddots & \vdots \\
c_{k-1,0}^\lambda & c_{k-1,1}^\lambda & \cdots & c_{k-1,k-1}^\lambda\\
\end{pmatrix}.\] 
Then, we may write the above equations more compactly as \[C^\lambda \cdot X = 0\text{ for all $\lambda \in 2\N$},\]
where $\cdot$ denotes the Frobenius inner product. We will show that the only symmetric matrix $X$ that satisfies all of these equations is the zero matrix.

%\bigskip

\begin{definition} Define the operator $T$ on $k\times k$ matrices, which moves the top row to the bottom and shifts every other row up by one:
\[T: \begin{pmatrix}
a_{0,0} & a_{0,1} & \cdots & a_{0,k-1} \\
a_{1,0} & a_{1,1} & \cdots & a_{1,k-1} \\
\vdots &\vdots & \ddots & \vdots\\
a_{k-1,0} & a_{k-1,1} & \cdots & a_{k-1,k-1} \\
\end{pmatrix} \mapsto 
\begin{pmatrix}
a_{1,0} & a_{1,1} & \cdots & a_{1,k-1} \\
a_{2,0} & a_{2,1} & \cdots & a_{1,k-1} \\
\vdots &\vdots & \ddots & \vdots\\
a_{k-1,0} & a_{k-1,1} & \cdots & a_{k-1,k-1} \\
a_{0,0} & a_{0,1} & \cdots & a_{0,k-1} \\
\end{pmatrix}. \]
\end{definition}
Note that $T^{-1}$ moves the bottom row to the top and shifts every other row down by one:
\[T^{-1}: \begin{pmatrix}
a_{0,0} & a_{0,1} & \cdots & a_{0,k-1} \\
a_{1,0} & a_{1,1} & \cdots & a_{1,k-1} \\
\vdots &\vdots & \ddots & \vdots\\
a_{k-1,0} & a_{k-1,1} & \cdots & a_{k-1,k-1} \\
\end{pmatrix} \mapsto 
\begin{pmatrix}
a_{k-1,0} & a_{k-1,1} & \cdots & a_{k-1,k-1} \\
a_{0,0} & a_{0,1} & \cdots & a_{0,k-1} \\
a_{1,0} & a_{1,1} & \cdots & a_{1,k-1} \\
\vdots &\vdots & \ddots & \vdots\\
a_{k-2,0} & a_{k-2,1} & \cdots & a_{k-2,k-1} \\
\end{pmatrix}. \]

We now use $T$ to prove the following theorem about $\spn_\R\left\{C^\lambda\right\}_{\lambda\in2\N}$. %{\color{red}maybe remark here that we also use k prime}
\begin{theorem}
Let $k$ be a prime number. It follows that,
\[\spn_\R\left\{C^\lambda\right\}_{\lambda\in2\N} = T\left(\Sym_k\right)\]
where $\Sym_k$ denotes the space of real $k\times k$ symmetric matrices.
\end{theorem}

\begin{proof} 
We first show $\spn_\R\left\{C^\lambda\right\}_{\lambda\in2\N} \subseteq T\left(\Sym_k\right)$. Fix $\lambda\in 2\mathbb{N}$ and let 
\[S^\lambda_{a,b} = \left\{\left(p,q\right)\in \mathbb{Z}_{\geq 0} \times \mathbb{Z}_{\geq 0}:
p \equiv a \mod{k},\,
q \equiv b \mod{k},\,
2q \left(p+1\right) = \lambda\right\}.\]
Note that $c^\lambda_{a,b} = \#S^\lambda_{a,b}$. The map \[(p,q) \mapsto (q-1,p+1)\]
is a bijection from $S^\lambda_{a,b}$ to $S^{\lambda}_{b-1,a+1}$. Therefore $c^\lambda_{a,b} = c^\lambda_{b-1,a+1}$ for all $a,b$. Note that \[T^{-1}\left(C^\lambda\right) = \begin{pmatrix}
c_{k-1,0}^{\lambda} & c_{k-1,1}^{\lambda} & c_{k-1,2}^{\lambda} & \cdots & c_{k-1,k-1}^{\lambda} \\
c_{0,0}^{\lambda} & c_{0,1}^{\lambda} & c_{0,2}^{\lambda} & \cdots & c_{0,k-1}^{\lambda} \\
c_{1,0}^{\lambda} & c_{1,1}^{\lambda} & c_{1,2}^{\lambda} &\cdots & c_{1,k-1}^{\lambda} \\
\vdots &  \vdots &\vdots & \ddots & \vdots \\
c_{k-2,0}^{\lambda} & c_{k-2,1}^{\lambda} & c_{k-2,2}^{\lambda} & \cdots & c_{k-2,k-1}^{\lambda} \\
\end{pmatrix}.\] 
Since $c^\lambda_{a,b} = c^\lambda_{b-1,a+1}$, the above matrix is symmetric. Hence, $T^{-1}\left(C^\lambda\right)$ is in $\Sym_k$ for each $\lambda$. This implies $C^\lambda$ is in $T\left(\Sym_k\right)$ for each $\lambda$.

We now show the other direction: $T\left(\Sym_k\right)\subseteq \spn_\R\left\{C^\lambda\right\}_{\lambda\in2\N}$. In particular, since $T$ is bijective, we show $\Sym_k\subseteq T^{-1}\left(\spn_\R\left\{C^\lambda\right\}_{\lambda\in2\N}\right)$. Let $E_{i,j}$ denote the $k\times k$ matrix whose $i,j$ entry is 1 and all other entries are 0. Note that $T^{-1}\left(E_{i,j}\right) = E_{\left(i+ 1\right) \% k , j}$ and the set 
\[\left\{E_{i,j} + E_{j,i} : 0\leq i \leq j \leq k-1\right\}\]
is a basis for $\Sym_k$. We show that each of these basis elements is in $T(\spn_\R\{C^\lambda\}_{\lambda\in2\N})$. Let $0\leq i\leq j\leq k-1$.

\textbf{Case 1:} Assume $0 = i = j $. Let $\lambda = 2k$. Since $k$ is prime, the only $\left(p,q\right)$ so that $2q\left(p+1\right) = \lambda$ are $\left(0,k\right)$ and $\left(k-1,1\right)$. This implies $C^{2k} = E_{0,0} + E_{k-1,1}$. Therefore, 
\begin{equation}
\label{eq:twok}
T^{-1}\left(C^{2k}\right) = E_{1,0} + E_{0,1}.
\end{equation}
Now let $\lambda = 2k^2$. Similarly, the only pairs $\left(p,q\right)$ so that $2q\left(p+1\right) = \lambda$ are $\left(0,k^2\right)$, $\left(k-1,k\right)$, and $\left(k^2-1,1\right)$. This implies $C^{2k} = E_{0,0} + E_{k-1,0} + E_{k-1,1}$. Therefore,
\begin{equation}
\label{eq:twoksquared}
T^{-1}\left(C^{2k^2}\right) = E_{1,0} + E_{0,0} + E_{0,1}.
\end{equation}
From (\ref{eq:twok}) and (\ref{eq:twoksquared}), it follows that \[
E_{0,0}  = T^{-1}\left(C^{2k^2}\right) - T^{-1}\left(C^{2k}\right).\]
Thus, the basis vector $E_{0,0} + E_{0,0}$ is in $T^{-1}\left(\spn_\R\{C^{\lambda}\}_{\lambda\in2\N}\right)$.

\textbf{Case 2:} Assume $0 = i < j$. Since $k$ is prime, $j$ is trivially coprime to $k$. By Dirichlet's theorem, there exists a prime $r$ so that $r \equiv j \mod k$. Let $\lambda = 2k r$. The only $\left(p,q\right)$ so that $2q\left(p+1\right) = \lambda$ are $\left(kr-1, 1\right)$, $\left(0, kr\right)$, $\left(k-1, r\right)$, and $\left(r-1, k\right)$. This implies,
\begin{equation}
\label{eq:twokr}
T^{-1}\left(C^{2kr}\right) = E_{0,1} + E_{1,0} + E_{0,j} + E_{j,0}.
\end{equation}
From (\ref{eq:twok}) and (\ref{eq:twokr}) \[E_{0,j} + E_{j,0} = T^{-1}\left(C^{2k r}\right) - T^{-1}\left(C^{2k}\right),\]
so $E_{0,j} + E_{j,0} \in T^{-1}\left(\spn_\R\{C^{\lambda}\}_{\lambda\in2\N}\right)$.

\textbf{Case 3:} Assume $0 < i \leq j$. Since $k$ is prime, $i$ and $j$ are coprime to $k$. Let $r$, $s$, and $t$ be primes such that $r \equiv j \mod k$, $s\equiv i \mod k$, and $t\equiv ij \mod k$. Let $\lambda = 2 rs$. Just as in case 2, we obtain the ordered pairs $\left(rs-1, 1\right)$, $\left(0, rs\right)$, $\left(s-1, r\right)$, and $\left(r-1, s\right)$, which implies
\begin{equation}
\label{eq:twors}
T^{-1}\left(C^{2rs}\right) = E_{ij\%k,1} + E_{1,ij\%k} + E_{i,j} + E_{j,i}.
\end{equation}

Finally, let $\lambda = 2t$. As in case 1, the only ordered pairs are $\left(t-1,1\right)$ and $\left(0,t\right)$, so 
\begin{equation}
\label{eq:twot}
T^{-1}\left(C^{2t}\right) = E_{ij\%k,1} + E_{1,ij\%k}.
\end{equation}

From (\ref{eq:twors}), and (\ref{eq:twot}), we see that \[E_{i,j} + E_{j,i} = T^{-1}\left(C^{2rs}\right) - T^{-1}\left(C^{2 t}\right),\]
so $E_{i,j} + E_{j,i} \in T^{-1}\left(\spn_\R\{C^{\lambda}\}_{\lambda\in2\N}\right)$.
\end{proof}

The following corollary is immediate when we regard the vector space of all real $k\times k$ matrices as an inner product space equipped with the Frobenius inner product, and when we notice that symmetric matrices and skew-symmetric matrices are orthogonal under this inner product and any $k\times k$ matrix is the sum of a symmetric matrix and skew-symmetric matrix.
\begin{corollary}
If a matrix $X$ solves the system 
\[C^\lambda \cdot X = 0\text{ for all }\lambda = 2,4,6,\ldots,\]
then $X\in T\left(\Sym_k\right)^\perp = T\left(\Skew_k\right)$, where $\Skew_k$ denotes the space of real $k\times k$ skew-symmetric (also known as antisymmetric) matrices.
\end{corollary}

%{\color{red}does this part of the proof deserves its own header/begintheorem/endtheorem?}

We now complete the proof of (2) implies (3). Assuming we have two CR isospectral lens spaces $L\left(k;\ell_1,\ell_2\right)$ and $L\left(k;\ell_1',\ell_2'\right)$ where $k$ is an odd prime, then there exists a symmetric matrix $X$ so that $C^\lambda\cdot X = 0$ for all $\lambda\in 2 \mathbb{N}$. By the above corollary, $X\in \Sym_k \cap T\left(\Skew_k\right)$. Equivalently, $T^{-1}\left(X\right)\in T^{-1}\left(\Sym_k\right) \cap \Skew_k$. Since $T^{-1}\left(X\right)\in\Skew_k$, we have 
\[T^{-1}\left(X\right) = - T^{-1}\left(X\right)^t.\] 
Applying $T$ to both sides yields, 
\[ X = - T\left(T^{-1}\left(X\right)^t\right).\]
Since $X$ is symmetric, transposing both sides imply
\[X = - T\left(T^{-1}\left(X\right)^t\right)^t.\]
Writing this equation in matrix form, we obtain
\[\begin{pmatrix}
x_{0,0} & x_{0,1} & x_{0,2}  & \cdots & x_{0,k-1} \\
x_{1,0} & x_{1,1} & x_{1,2}  & \cdots & x_{1,k-1} \\
x_{2,0} & x_{2,1} & x_{2,2}  & \cdots & x_{2,k-1} \\
\vdots &\vdots & \vdots & \ddots & \vdots\\
x_{k-1,0} & x_{k-1,1} & x_{k-1,2} & \cdots & x_{k-1,k-1} \\
\end{pmatrix} = 
\begin{pmatrix}
x_{k-1,1} & x_{k-1,2} & x_{k-1,3}  & \cdots & x_{k-1,k-1} & x_{k-1,0} \\
x_{0,1} & x_{0,2} & x_{0,3}  & \cdots & x_{1,k-1} & x_{1,0} \\
x_{1,1} & x_{1,2} & x_{1,3}  & \cdots & x_{1,k-1} & x_{2,0} \\
\vdots &\vdots & \vdots & \ddots & \vdots & \vdots\\
x_{k-2,1} & x_{k-2,2} & x_{k-2,3} & \cdots & x_{k-2,k-1} & x_{k-2,0} \\
\end{pmatrix}.\]
This means $x_{a,b} = -x_{\left(a-1\right)\% k,\left(b+1\right)\% k}$ for all $0\leq a,b\leq k-1$. Applying this fact $k$ times, we see that \[x_{a,b} = \left(-1\right)^k x_{\left(a-k\right)\% k,\left(b+k\right)\% k} = \left(-1\right)^k x_{a,b} = - x_{a,b}.\]
Therefore, we must have $x_{a,b} = 0$ for all $a,b$. Thus, $X = 0$. As noted above, this implies $\dim\calH_{p,q}^G = \dim\calH_{p,q}^{G'}$ for all $p,q$, completing the proof that (2) implies (3) in Theorem \ref{thm:isospec}.
%\footnote{{\color{red} alt pf method: }
%consider $T$ as a $k\times k$ matrix; \[T = \begin{pmatrix}
%0 & 1 & 0 & 0 &\cdots& 0\\
%0 & 0 & 1 & 0 &\cdots& 0\\
%\vdots & \vdots &\vdots & \ddots &  & \vdots\\
%0& 0 & 0 & 0 & \cdots& 1\\
%1& 0 & 0 & 0 &\cdots&0\\
%\end{pmatrix}\]
%then $T(X)$ is simply the matrix product $TX$. (we do this so that it makes sense to talk about expressions like $XT$).  Note that $T^t = T^{-1}$. Since $TX$ is skew-symmetric, $(TX)^t = - TX$, which gives 
%\begin{align*}
%X^t T^t = -T X\\
%X^t T^{-1} = -T X\\
%X = - T X T
%\end{align*}
%Substituting $- TXT$ for $X$ $k$ times gives the equation \[X = (-1)^k T^k X T^k.\]
%Since $T^k = 1$ and $k$ is odd, this becomes $X = -X$, meaning $X = 0$.}

\subsection{(3) implies (4)} 

In this section, we extend techniques from \cite{ikeda1979} to the CR setting. For a lens space $L(k;\ell_1,\ldots,\ell_n)$ corresponding to group $G$, consider the generating function below taking inputs as pairs of complex numbers with sufficiently small modulus,
\[F_{\left(k; \ell_1,\ldots,\ell_n\right)}\left(z,w\right) = \sum_{p,q\geq 0} \left(\dim\calH_{p,q}^G\right)z^p w^q.\]
We now present a closed form for this generating function that connects the CR spectrum of this lens space and its geometry.
\begin{theorem} 
\label{thm:genfunc}
Let $G$ be the group corresponding to the lens space $L\left(k;\ell_1,\ldots,\ell_n\right)$. It follows that,
\[F_{\left(k; \ell_1,\ldots,\ell_n\right)}\left(z,w\right)  = \frac1k\sum_{m = 0}^{k-1} \frac{1 - z w} {\prod_{i = 1}^n\left(z-\zeta^{-m\ell_i}\right)\left(w-\zeta^{m\ell_i} \right)}\]
where $\zeta = e^{2\pi i/k}$.
\end{theorem}
\begin{proof}
Let $\calP_{p,q}$ denote the space of polynomials of bidegree $p,q$ on $S^{2n-1}$. We may consider $\calH_{p,q}$ and $\calP_{p,q}$ as $G$-modules over $\mathbb{C}$. Note that $\calP_{p,q} \cong \calH_{p,q}\oplus |z|^2\calP_{p-1,q-1}$ \cite{klima2004}.

Let $\chi_{p,q}$ be the character of $\calH_{p,q}$ and $\widetilde{\chi}_{p,q}$ be the character of $\calP_{p,q}$. Note that \[\left\{z^{\alpha}\overline{z}^\beta : \alpha,\beta\in \Z_{\geq 0}^n, \left|\alpha\right| = p, \left|\beta\right| = q\right\}\]
is a basis for $\calP_{p,q}$, and recall that
\[g^m \cdot z^{\alpha} \overline{z}^\beta = \zeta^{m \left(\sum_{i=1}^n{\ell_i \left(\alpha_i - \beta_i\right)}\right)} z^{\alpha} \overline{z}^\beta.\]
Since this is a basis on which $g^m$ acts diagonally, we see that \[\widetilde{\chi}_{p,q}\left(g^m\right) = \sum_{\left|\alpha\right| = p,\left|\beta\right| = q}  \zeta^{m \left(\sum_{i=1}^n{\ell_i \left(\alpha_i - \beta_i\right)}\right)}.\]
Notice that in 
\[\prod_{i = 1}^n \left(1 + \zeta^{m \ell_i}z + \zeta^{2m \ell_i}z^2 + \cdots\right) \left(1 + \zeta^{-m \ell_i}w + \zeta^{-2m \ell_i}w^2 + \cdots\right),\]
the coefficient of $z^p w^q$ is precisely $\widetilde{\chi}_{p,q}\left(g^m\right)$. It follows that,
\[\sum_{p,q\geq 0}\widetilde{\chi}_{p,q}\left(g^m\right) z^p w^q = 
 \frac{1}{\prod_{i = 1}^n\left(1 - \zeta^{m \ell_i}z\right) \left(1 - \zeta^{-m \ell_i}w\right)},\]
where we have used the fact that $1 + x + x^2 + \cdots = \left(1-x\right)^{-1}$.

We now relate the characters to the dimension of eigenspaces. Since $\calP_{p,q} \cong \calH_{p,q} \oplus |z|^2\calP_{p-1,q-1}$, it follows that \[\chi_{p,q} =  \widetilde\chi_{p,q}- \widetilde{\chi}_{p-1,q-1}.\]
Note that when $p$ or $q$ is zero $\calP_{p,q}=\calH_{p,q}$; therefore, we set $\widetilde{\chi}_{0,-1}=\widetilde{\chi}_{-1,0}=\widetilde{\chi}_{-1,-1}=0$.
Moreover, $\dim\calH_{p,q}^G$ can be obtained by averaging the character $\chi_{p,q}$  over the group $G$ as follows, 
\[\dim \mathcal{H}^G_{p,q} = \frac1k \sum_{m = 0}^{k-1} \chi_{p,q}\left(g^m\right).\]
Combining all of these facts, we see that 
\begin{align*}
\sum_{p,q\geq 0} \left(\dim\mathcal{H}_{p,q} ^G\right) z^p w^q 
&= \frac1k\sum_{p,q\geq 0} \lrp{\sum_{m = 0}^{k-1} \chi_{p,q}\left(g^m\right)} z^p w^q\\
&= \frac1k\sum_{p,q\geq 0} \lrp{\sum_{m = 0}^{k-1} \widetilde\chi_{p,q}\left(g^m\right) - \widetilde\chi_{p-1,q-1}\left(g^m\right)} z^p w^q\\
&= \frac1k\sum_{m = 0}^{k-1}\lrp{\lrp{ \sum_{p,q\geq 0} \widetilde\chi_{p,q}\left(g^m\right)z^p w^q} - z w \lrp{\sum_{p,q\geq 1}\widetilde\chi_{p-1,q-1}\left(g^m\right) z^{p-1} w^{q-1}}}\\
&= \frac1k\sum_{m = 0}^{k-1} \frac{1-zw} {\prod_{i = 1}^n\left(1-\zeta^{m\ell_i} z\right)
\left(1-\zeta^{-m\ell_i} w\right)}= \frac1k\sum_{m = 0}^{k-1} \frac{1-zw} {\prod_{i = 1}^n\left(z-\zeta^{-m\ell_i}\right)
\left(w-\zeta^{m\ell_i}\right)}.
\end{align*}
\end{proof}
To complete the proof that $(3)$ implies $(4)$ in Theorem \ref{thm:isospec} in the $n=2$ case, we need the following.

\begin{lemma} 
\label{lem:independence} 
Let $\zeta = e^{2\pi i/k}$. The set 
\[\left\{\frac{1}{\left(z-\zeta^i \right)\left(w-\zeta^{-i}\right)\left(z-\zeta^j \right)\left(w-\zeta^{-j}\right)}\ \middle |\ 0\leq i\leq j\leq k-1\right\}\]
is linearly independent over $\C$.
\end{lemma}

\begin{proof}
We will label the elements of the above set as
\[f_{l,m} = \frac1{\left(z-\zeta^l\right)\left(w-\zeta^{-l}\right)\left(z-\zeta^m\right)\left(z-\zeta^{-m}\right)}.\]
First, note that $\operatorname{span}\left\{f_{l,l}: 0 \leq l < k-1\right\} \cap \operatorname{span}\left\{f_{l,m}: 0 \leq l < m \leq k-1\right\} = \left\{0\right\}$
and the set $\left\{f_{l,l}: 0 \leq l < k-1\right\}$
is linearly independent.
Now for $\left\{f_{l,m}: 0 \leq l < m \leq k-1\right\}$,
suppose there exists $a_{l,m}\in \mathbb{C}$ so that
\[\sum_{l=0}^{k-1}\sum_{m=l+1}^{k-1}a_{l,m}f_{l,m}=\sum_{l=0}^{k-1}\sum_{m=l+1}^{k-1}\frac{a_{l,m}}{\left(z-\zeta^l\right)\left(w-\zeta^{-l}\right)\left(z-\zeta^m\right)\left(w-\zeta^{-m}\right)}=0.\]
If we multiply by a factor of $\left(z^k-1\right)\left(w^k-1\right)$,
then we are left with a polynomial
\[\left(z^k-1\right)\left(w^k-1\right)\sum_{l=0}^{k-1}\sum_{m=l+1}^{k-1}\frac{a_{l,m}}{\left(z-\zeta^l\right)\left(w-\zeta^{-l}\right)\left(z-\zeta^m\right)\left(w-\zeta^{-m}\right)}=0.\]
Setting $z=\zeta^{l_0}$, $w=\zeta^{-m_0}$, where $0 \leq l_0 < m_0 \leq k-1$ implies $a_{l_0,m_0}=0$. Therefore, the set $\left\{f_{l,m}: 0 \leq l < m \leq k-1\right\}$ is linearly independent.
\end{proof}

\begin{corollary}
Let $L\left(k; \ell_1,\ell_2\right)$ and $L\left(k; \ell_1',\ell_2'\right)$ be lens spaces with groups $G$ and $G'$ respectively. If $\dim\calH^G_{p,q} = \dim\calH^{G'}_{p,q}$ for all $p,q$, then there exists a permutation $\sigma$ and an integer $a$ such that $\left(\ell_1,\ell_2\right) = \left(a\ell'_{\sigma\left(1\right)},a\ell'_{\sigma\left(2\right)}\right)$.
\end{corollary}
\begin{proof}
Since $\dim\calH^G_{p,q} = \dim\calH^{G'}_{p,q}$ for all $p,q$, $F_{\left(k; \ell_1,\ell_2\right)}\left(x\right) = F_{\left(k; \ell_1',\ell_2'\right)}\left(x\right)$. Applying Theorem~\ref{thm:genfunc} and cancelling like terms, we obtain 
%\begin{align*}
\[\sum_{m = 0}^{k-1} \frac{1} 
{\left(z-\zeta^{-m\ell_1}\right)\left(w-\zeta^{m\ell_1}\right)\left(z-\zeta^{-m\ell_2}\right)\left(w-\zeta^{m\ell_2}\right)} =
\sum_{m = 0}^{k-1} \frac{1} 
{\left(z-\zeta^{-m\ell_1'}\right)\left(w-\zeta^{m\ell_1'}\right)
\left(z-\zeta^{-m\ell_2'}\right)\left(w-\zeta^{m\ell_2'}\right)}.\]
%\end{align*} 
By Lemma \ref{lem:independence}, the terms of these sums are linearly independent, so we may equate each term on the left with a term on the right. By considering the $m = 1$ term on the left, there exists some $m_0$ such that 
%\begin{align*}
\[\frac{1}{\left(z-\zeta^{-\ell_1}\right)\left(w-\zeta^{\ell_1}\right)
\left(z-\zeta^{-\ell_2}\right)\left(w-\zeta^{\ell_2}\right)}  
= \frac{1}{\left(z-\zeta^{-m_0\ell_1'}\right)\left(w-\zeta^{m_0\ell_1'}\right)
\left(z-\zeta^{-m_0\ell_2'}\right)\left(w-\zeta^{m_0\ell_2'}\right)}.\]
%\end{align*}
This implies $\left(m_0\ell_1',m_0\ell_2'\right)$ is some permutation of $\left(\ell_1, \ell_2\right)$. The claim follows after setting $a = m_0$.
\end{proof}
\subsection{(4) implies (1) and (1) implies (2)}

We begin with the following theorem which shows that (4) implies (1).

\begin{theorem}
Let $L\left(k;\ell_1,\ldots,\ell_n\right)$ and $L\left(k;\ell_1',\ldots,\ell_n'\right)$ be lens spaces. If there exists a permutation $\sigma$ and an integer $a$  such that $\left(\ell_1',\ldots,\ell_n'\right) \equiv \left(a\ell_{\sigma\left(1\right)},\ldots,a\ell_{\sigma\left(n\right)}\right) \mod k$, then $L\left(k;\ell_1,\ldots,\ell_n\right)$ and $L\left(k;\ell_1',\ldots,\ell_n'\right)$ are CR isometric.
\end{theorem}
\begin{proof}
For any permutation $\sigma$, \[\left(z_1,z_2,\dots,z_n\right) \mapsto (z_{\sigma\left(1\right)},z_{\sigma\left(2\right)}, \dots, z_{\sigma\left(n\right)})\]
is a CR isometry from $S^{2n-1}$ to itself. This induces a CR isometry from $L\left(k; \ell_1, \dots, \ell_n\right)$ to $L\left(k; \ell_{\sigma(1)}, \dots, \ell_{\sigma\left(n\right)}\right)$. 

Note that $a$ must be relatively prime to $k$ by the definition of a lens space. If $g$ denotes the action corresponding to $L\left(k; \ell_1, \dots, \ell_n\right)$, then $g^a$ generates the same subgroup as $g$. Hence $L\left(k; \ell_1, \dots, \ell_n\right) = L\left(k; a\ell_1, \dots, a\ell_n\right)$, which completes the proof.
\end{proof}

Finally, we need to prove (1) implies (2) to conclude the proof of Theorem \ref{thm:isospec}. This implication follows from a more general statement that the Kohn Laplacian commutes with CR isometries and therefore isometric CR manifolds are isospectral. We refer to \cite[Section 4.4]{Canzani} for the proof in the Riemannian setting and argue similarly in the CR setting.

\newcommand{\etalchar}[1]{$^{#1}$}

%\bibliographystyle{alpha}
%\bibliography{sources} 

\begin{thebibliography}{ABB{\etalchar{+}}19}

\bibitem[ABB{\etalchar{+}}19]{REU18}
John Ahn, Mohit Bansil, Garrett Brown, Emilee Cardin, and Yunus~E. Zeytuncu.
\newblock Spectra of {K}ohn {L}aplacians on spheres.
\newblock {\em Involve}, 12(5):855--869, 2019.

\bibitem[BGS{\etalchar{+}}21]{REU2020Weyl}
Henry Bosch, Tyler Gonzales, Kamryn Spinelli, Gabe Udell, and Yunus~E.
  Zeytuncu.
\newblock {A Tauberian approach to Weyl’s law for the Kohn Laplacian on
  spheres}.
\newblock {\em Canadian Mathematical Bulletin}, page 1–21, 2021.

\bibitem[Can13]{Canzani}
Yeiza Canzani.
\newblock {\em Notes for Analysis on Manifolds via the Laplacian}.
\newblock Lecture Notes.
  https://canzani.web.unc.edu/wp-content/uploads/sites/12623/2016/08/Laplacian.pdf,
  2013.

\bibitem[Fol72]{Folland}
G.~B. Folland.
\newblock The tangential {C}auchy-{R}iemann complex on spheres.
\newblock {\em Trans. Amer. Math. Soc.}, 171:83--133, 1972.

\bibitem[IY79]{ikeda1979}
Akira Ikeda and Yoshihiko Yamamoto.
\newblock On the spectra of {$3$}-dimensional lens spaces.
\newblock {\em Osaka Math. J.}, 16(2):447--469, 1979.

\bibitem[Kli04]{klima2004}
Oldrich Klima.
\newblock Analysis of a subelliptic operator on the sphere in complex n-space.
\newblock Master's thesis, The University of New South Wales, April 2004.

\bibitem[ST84]{Stanton1984TheHE}
N.~K. Stanton and David~S. Tartokoff.
\newblock The {H}eat {E}quation for the $\overline{\partial}_b$–{L}aplacian.
\newblock {\em Communications in Partial Differential Equations}, 9:597--686,
  1984.

\bibitem[Sta84]{Stanton}
Nancy~K. Stanton.
\newblock {The heat equation in several complex variables}.
\newblock {\em Bulletin (New Series) of the American Mathematical Society},
  11(1):65 -- 84, 1984.

\bibitem[Str15]{strichartz2015}
Robert Strichartz.
\newblock Spectral asymptotics on compact {H}eisenberg manifolds.
\newblock {\em The Journal of Geometric Analysis}, 26, 07 2015.

\end{thebibliography}

\end{document}